\documentclass[12pt,a4paper]{amsart}
\usepackage{tgtermes}

\usepackage[margin=2cm,a4paper]{geometry}
\usepackage[foot]{amsaddr}
\usepackage{enumitem}
\usepackage{lineno}
\usepackage{hyperref}
\hypersetup{
    colorlinks=true,
    citecolor=blue,
    linkcolor=blue,
    filecolor=blue,      
    urlcolor=blue,
    pdftitle={On the chromatic number of the union of comparability graphs},
    pdfauthor={Wouter Cames van Batenburg, Maria Chudnovsky, Linda Cook, James Davies, Seokbeom Kim, Sang-il Oum}
}
\usepackage{keytheorems}
\usepackage{zref-clever}
\zcsetup{cap=true}
\newkeytheorem{theorem}[name=Theorem,refname={Theorem,Theorems}, Refname={Theorem,Theorems}]
\newkeytheorem{lemma,proposition,corollary,conjecture,problem}[sibling=theorem]
\newcommand\abs[1]{\left\lvert #1\right\rvert}

\begin{document}
\title{On the chromatic number of 
the union of comparability~graphs}
\date{\today}
\author{Wouter Cames van Batenburg$^1$}
\address{$^1$D\'epartement d'Informatique, Universit\'e libre de Bruxelles, Belgium.}

\email{w.p.s.camesvanbatenburg@gmail.com}
\author{Maria Chudnovsky$^2$}
\address{$^2$Department of Mathematics, Princeton University, Princeton, USA.}
\email{mchudnov@math.princeton.edu}

\author{Linda Cook$^3$}
\address{$^3$Mathematics Institute, Utrecht University, Utrecht, the Netherlands.}
\email{l.j.cook@uu.nl}

\author{James Davies$^4$}
\address{$^4$Institute of Mathematics, Leipzig University, Leipzig, Germany.}
\email{jgdavies@uwaterloo.ca}

\author{Seokbeom Kim$^{6,5}$}
\email{seokbeom@kaist.ac.kr}

\author{Sang-il Oum$^{5,6}$}
\address{$^5$Discrete Mathematics Group, Institute for Basic Science (IBS), Daejeon,~South~Korea.}
\address{$^6$Department of Mathematical Sciences, KAIST, Daejeon, South Korea.}
\thanks{
WC was supported by the Belgian National Fund for Scientific Research (FNRS).
MC was supported by NSF Grant DMS-2348219, NSF Grant CCF-2505100, AFOSR grant FA9550-25-1-0275, and a Guggenheim Fellowship.
JD was supported by the Alexander von Humboldt Foundation in the framework of the Alexander von Humboldt Professorship of Daniel Král' endowed by the Federal Ministry of Education and Research.
SK and SO were supported by the Institute for Basic Science (IBS-R029-C1).
}
\email{sangil@ibs.re.kr}

\begin{abstract}
    Resolving in a strong sense a problem of Gy\'arf\'as on the union of two perfect graphs, we prove that for every pair of positive integers $d$ and $k$, there is a graph~$G$ with clique number $k$ and chromatic number $k^d$ that is the union of $d$ comparability graphs.
    We also show that the chromatic number can be replaced by the fractional chromatic number or $\frac{\abs{V(G)}}{\alpha(G)}$.
\end{abstract}
\maketitle

\section{Introduction}\label{sec:intro}
For a graph $G$, let us write $\omega(G)$ for the \emph{clique number}, that is, the maximum size of a clique in $G$, 
and $\chi(G)$ for the chromatic number of $G$.
A graph is \emph{perfect} if $\chi(H)=\omega(H)$ for all of its induced subgraphs $H$.
We say a graph $G$ is the \emph{union} of graphs $G_1, \ldots, G_d$
if $V(G)=V(G_1)\cup \cdots \cup V(G_d)$ and $E(G)
=E(G_1)\cup 
\cdots \cup E(G_d)$.
By taking a product coloring, if $G$ is the union of $d$ perfect graphs, 
then $\chi(G)\le (\omega(G))^d$.
This raises the natural problem of determining the optimal $\chi$-bounding function for the class of graphs that are unions of $d$ perfect graphs. For $d=2$, this is a problem of Gy\'arf\'as~\cite[Problem 5.3]{Gyarfas1987} from 1985.

A \emph{comparability graph} is a graph on a partially ordered set $(P,\preceq)$ where two vertices are adjacent if and only if they are comparable. 
By Dilworth's theorem~\cite{Dilworth1950}, comparability graphs are perfect.
There are natural geometric classes of disjointness graphs that are the union of comparability graphs~\cite{fox2008erdHos}. For instance, disjointness graphs of grounded $x$-monotone curves are the union of $2$ comparability graphs~\cite{PT2020}, disjointness graphs of arbitrary convex sets in the plane are the union of $4$ comparability graphs~\cite{larman1994ramsey}, and disjointness graphs of axis-aligned boxes in $\mathbb{R}^d$ are the union of $d$ comparability graphs~\cite{cozzens1983computing}.

We resolve Gy\'arf\'as' problem~\cite[Problem 5.3]{Gyarfas1987} in a strong sense by showing that $(\omega(G))^d$ is the optimal $\chi$-bounding function, even if $G$ is the union of $d$  comparability graphs.

\begin{theorem}[store=main]\label{thm:main}
    For positive integers $k$ and $d$, 
    there is a graph $G$ that is the union of $d$ comparability graphs such that 
    $\omega(G)=k$ and $\chi(G)=k^d$.
\end{theorem}

\zcref{thm:main} is tight and improves previous lower bound constructions;
Dumitrescu and T\'oth~\cite{DT2002} showed that there are such graphs $G$ with $
\chi(G)\ge (\omega(G))^{d(1-o(1))}$, and Pach and Tomon~\cite[Theorem 14]{PT2020} constructed unions of two comparability graphs with chromatic number $\binom{\omega(G)+1}{2}$.
For further Ramsey results on the union of comparability graphs, see \cite{bradavc2024clique,DT2002,fox2009bipartite,fox2026multipartite,KT2020,tomon2016turan}.

\zcref{thm:main} can be strengthened for 
\emph{fractional chromatic number} $\chi_f(G)$ and $\frac{\abs{V(G)}}{\alpha(G)}$, where $\alpha(G)$ denotes the maximum size of an independent set in~$G$.
Note that $\frac{\lvert V(G) \rvert}{\alpha(G)} \leq \chi_f(G) \leq \chi(G)$ in general.

\begin{theorem}\label{thm:main_Hallratio}
    For positive integers $k$ and $d$, there is a graph $G$ that is the union of $d$ comparability graphs such that $\omega(G) = k$ and $\frac{\lvert V(G) \rvert}{\alpha(G)}=\chi_f(G)=\chi(G) = k^d$.
\end{theorem}

\section{The construction}

For a positive integer $k$, we let $[k]=\{1, \ldots , k\}$.
The \emph{girth} of a graph is the length of its shortest cycle. 
Our construction uses graphs with arbitrarily large girth and chromatic number.
Erd{{\H{o}}}s~\cite{erdos59-high-girth-high-chromatic-graph} first showed that such graphs exist using the probabilistic method and Lov\'{a}sz \cite{lovasz-explicit} gave the first explicit construction.
For further explicit constructions, see \cite{noga-high-girth,bucic2025geometric,Davies2021Box,hypergraphs-high-chromatic-and-girth3,hypergraphs-high-chromatic-and-girth4,ramanujan,hypergraphs-high-chromatic-and-girth5}.

\begin{theorem}[Erd{{\H{o}}}s \cite{erdos59-high-girth-high-chromatic-graph}]\label{large_girth_large_chi}
    For every pair of positive integers $g$ and $k$, there exists a graph~$G$ of girth more than $g$ and with $\chi(G)=k$.
\end{theorem}

The idea of how to partition into comparability graphs is roughly based on Pach and Tomon's~\cite[Theorem 14]{PT2020} aforementioned construction of the union of 2 comparability graphs with chromatic number $\binom{\omega(G)+1}{2}$, which was probabilistic.
Instead of a probabilistic approach, we also use a trick of starting with a graph with large girth and fixed chromatic number. A somewhat similar trick to how we use these graphs was also pointed out by Ne{\v{s}}et{\v{r}}il  for a different purpose (see \cite{girao2024induced}).
Our construction is explicit and it can also be modified so that for any positive integer~$h$, $G$ contains no hole of length at most $h$.

\begin{proof}[Proof of~\zcref{thm:main}]
    We may assume $k\ge2$.
    By~\zcref{large_girth_large_chi}, there is a graph $G_0$ with girth at least $3k-2$ and chromatic number equal to $k^d$.
    We let $\phi: V(G_0) \to [k]^d$ be a fixed $k^d$-coloring of~$G_0$.
    For each $1\le r \le d$, let $(V(G_0),\preceq_r)$ be the partial order such that $u \preceq_r v$ if $G_0$ has a path $p_1\cdots p_s$ for $s\ge1$ with $p_1=u$, $p_s=v$ such that for every $1< i \le s$, the first $r-1$ coordinates of $\phi(p_{i-1})$ and $\phi(p_{i})$ are equal and the $r$-th coordinate of $\phi(p_i)$ is strictly greater than the $r$\nobreakdash-th coordinate of $\phi(p_{i-1})$.
    Note that such a path $p_1\cdots p_s$ forms a chain of $(V(G_0),\preceq_r)$, and that $s\le k$.
    Let $H_r$ be the comparability graph of $(V(G_0),\preceq_r)$.
    Then, $H_r$ has clique number at most $k$ since $(V(G_0),\preceq_r)$ clearly has no chain of length $k+1$.
    Since comparability graphs are perfect, we have that $\chi(H_r)=\omega(H_r)\le k$.
    Let $G=H_1 \cup \cdots \cup H_d$.
    Then by considering a product coloring, we have that $\chi(G) \le \omega(H_1)\cdots\omega(H_d)\le k^d$.
    For every edge~$uv$ of~$G_0$, we have that $u\prec_r v$ 
    or~$v\prec_r u$
    for some unique $1\le r\le d$, and in particular that $uv$ is an edge of~$H_r$.
    Therefore, $G_0$ is a subgraph of~$G$, and so $\chi(G)\ge \chi(G_0) =k^d$.
    Hence, $\chi(G)=k^d$.
    Since $\omega(H_1),\ldots , \omega(H_d)\le k$ and $k^d=\chi(G) \le \omega(H_1)\cdots\omega(H_d)$, it follows that $\omega(H_1) = \cdots = \omega(H_d) = k$, and therefore that $\omega(G)\ge k$.

    It remains to show that $\omega(G)\le k$.
    Suppose for sake of contradiction that $\omega(G)>k$.
    Then~$G$ contains a triangle on vertices $u,v,w$ such that the edges $uv,uw,vw$ are not all contained in a single~$H_r$.
    In particular, this implies without loss of generality that there exist $1\le a,b,c \le d$ with $a,b,c$ not all equal such that $u\prec_a v$, 
    $v \prec_b w$,
    and either 
    $u\prec_c w$
    or 
    $w\prec_c u$.
    Therefore, $G_0$ has paths $x_1\cdots x_{s_1}$, $y_1\cdots y_{s_2}$, $z_1\cdots z_{s_3}$ such that $x_1=u=z_{s_3}$, $x_{s_1}=v=y_1$, $y_{s_2}=w=z_1$, 
    $x_1 \prec_a \cdots \prec_a x_{s_1}$, $y_1 \prec_b \cdots \prec_b y_{s_2}$, 
    and 
    either 
    $z_1 \prec_c  \cdots \prec_c z_{s_3}$
    or 
    $z_{s_3} \prec_c  \cdots \prec_c z_{1}$.
    The ends of every edge of $G_0$ are comparable in exactly one of the partial orders ${(V(G_0),\preceq_1)}, \ldots , (V(G_0),\preceq_d)$.
    Since $a,b,c$ are not all equal, it therefore follows that one of the paths $x_1\cdots x_{s_1}$, $y_1\cdots y_{s_2}$, $z_1\cdots z_{s_3}$ is edge-disjoint from the other two.
    As $s_1,s_2,s_3 \le k$, it therefore follows that the union of these three paths contains a cycle of length at most $3k-3$.
    Since this cycle is contained in $G_0$, this contradicts our choice of~$G_0$ having girth at least $3k-2$.
    Hence $\omega(G)=k$, as desired.
\end{proof}

To prove \zcref{thm:main_Hallratio}, we use the following strengthening of
\zcref{large_girth_large_chi}, due to Lin and Zhu~\cite{LZ06}. 
Although their result explicitly states that $\chi_f(G)=\chi(G)=h$, the construction in their proof also yields $\frac{\abs{V(G)}}{\alpha(G)}=h$.
For completeness, we include an alternative proof in the appendix.
\begin{theorem}[store=chif]\label{thm:largegirthlargeHallratio}
    For every pair of positive integers $g$ and $h$, there exists a graph $H=H_{g, h}$ such that the girth of $H$ is more than $g$ and $\frac{\lvert V(H) \rvert}{\alpha(H)} = \chi_f(H) = \chi(H) = h$.
\end{theorem}

\begin{proof}[Proof of \zcref{thm:main_Hallratio}]    
    Let $G_0 = H_{3k-2, k^d}$ be the graph given by \zcref{thm:largegirthlargeHallratio}.
    Since adding edges does not decrease the inverse independence ratio, we obtain a graph $G$  that is the union of $d$ comparability graphs 
    such that $\omega(G) = k$ and $\frac{\abs{V(G)}}{\alpha(G)} = k^d$
    by the proof of~\zcref{thm:main} starting with $G_0$.
\end{proof}

We leave as an open problem to find a constructive proof for \zcref{thm:largegirthlargeHallratio}. 
This would immediately yield a construction for \zcref{thm:main_Hallratio} by substituting the construction from \zcref{thm:largegirthlargeHallratio} into the proof of \zcref{thm:main_Hallratio}.

\subsection*{Tool and computational resource disclosure}
The authors used ChatGPT 5.5 Pro to find a proof of Pach and Tomon~\cite[Theorem 14]{PT2020} without using the probabilistic method, which allowed the authors to simplify the proof in the present paper.

\bibliographystyle{amsplain}
{\providecommand{\bysame}{\leavevmode\hbox to3em{\hrulefill}\thinspace}
\providecommand{\MR}{\relax\ifhmode\unskip\space\fi MR }
\providecommand{\MRhref}[2]{\href{http://www.ams.org/mathscinet-getitem?mr=#1}{#2}
}
\providecommand{\href}[2]{#2}

}

\appendix
\section{Graphs of large girth and prescribed Hall ratio}\label{sec:largegirthlargeHallratio}

We present a direct proof of \zcref{thm:largegirthlargeHallratio}.
Let us denote by $\operatorname{girth}(G)$ the girth of $G$.
We will prove the following lemma.

\begin{lemma}\label{lem:partite}
Let $h \ge 3$ and $g \ge 1$ be fixed integers. 
Then there exist an integer $m \geq 1$ and an $h$-partite graph $G$ with parts $V_1, \ldots, V_h$, all of the same size $m$, such that the girth of $G$ is more than $g$ and $\alpha(G)=m$.
\end{lemma}

Assuming \zcref{lem:partite}, we can prove \zcref{thm:largegirthlargeHallratio}.
\begin{proof}[Proof of~\zcref{thm:largegirthlargeHallratio} assuming \zcref{lem:partite}]
    Let $H_{g,1}=K_1$ and let $H_{g,2}$ be an even cycle of length greater than $g$. 
    For $h\ge 3$, by~\zcref{lem:partite}, we can choose $H=H_{g,h}$ to be an $h$-partite graph such that all its parts have size $m$, its girth is greater than $g$, and $\alpha(H)=m$. 
    Then we have
    $h=\abs{V(H)}/\alpha(H) \leq \chi_f(H) \le \chi(H) \le h$.
    This proves~\zcref{thm:largegirthlargeHallratio}.
\end{proof}
\begin{proof}[Proof of \zcref{lem:partite}]
    Fix an integer $\Lambda$ so that $e^\Lambda> 2e^{2h}$ and $2^{\Lambda}>4e^2h$.
    For a sufficiently large integer $m$ to be specified later, let $V_i=\{(i,x):x\in[m]\}$ for each $i \in [h]$.
    We will construct a random subgraph $G$ of the complete multipartite graph with parts $V_1, V_2, \ldots, V_h$, as follows. 
    For every ordered pair $(i, j)$ such that $i, j \in [h]$ and $i \neq j$, every $\lambda\in[\Lambda]$, and every $x\in[m]$, choose $f_{ij}^\lambda(x)\in[m]$ independently and uniformly, add a directed edge from $(i,x)$ to $(j,f_{ij}^{\lambda}(x))$, and label the directed edge with $\lambda$. 
    In other words, we draw $\Lambda$ random labelled directed edges from $(i,x)$ to $V_j$, where parallel edges are allowed. 
    This yields a directed multigraph $\Gamma$. 
    After this, construct $G$ from $\Gamma$ by putting an undirected edge between two vertices whenever at least one labelled directed edge joins them, ignoring multiple edges. This yields a simple $h$-partite graph $G$ on $mh$ vertices.

    We first show that, with probability bounded below by a positive constant depending only on $h$, $\Lambda$, and $g$, there is no cycle of length at most $g$ in $G$.
    For $3\le \ell\le g$, consider a potential labelled oriented cycle of length $\ell$, by which we mean $\ell$ distinct vertices $(i_1,x_1),\ldots,(i_\ell,x_\ell)$ with $i_t\ne i_{t+1}$ cyclically, together with one of the two directions and a label $\lambda_t\in[\Lambda]$ between each consecutive pair of vertices. 
    A \emph{bad event} is that all these labelled directed edges occur in $\Gamma$. 
    Observe that every actual cycle of $G$ of length at most $g$ originates from at least one such bad event.

    Let $B$ be a fixed potential labelled oriented cycle of length $\ell$ in $\Gamma$.
    Since $B$ involves $\ell$ distinct random variables, it holds that $\mathbb{P}(B) \leq m^{-\ell}$.
    In addition, the number of potential labelled oriented cycles of length $\ell$ is at most a constant multiple of $m^{\ell}$.
    Indeed, to specify such a potential labelled cycle, we have at most $h^{\ell}$ choices for the $\ell$ parts, at most $m^{\ell}$ choices for the vertices $x_1,\ldots, x_{\ell}$, and $2 \Lambda$ choices for the direction and the label for each consecutive pair. 
    This shows that there are at most $h^\ell m^\ell(2\Lambda)^\ell=O_{h,\Lambda,\ell}(m^{\ell})$ potential labelled oriented cycles of length $\ell$ in $G$.

    We use the standard asymmetric Lov\'asz local lemma (see, for example,~\cite[Chapter~5]{AlonSpencer}).
    Let~$\mathcal{B}$ be a family of bad events with a dependency graph.
    If there are numbers $0<x_B<1$ such that $\mathbb{P}(B)\le x_B\prod_{B'\sim B}(1-x_{B'})$ for every $B \in \mathcal{B}$, then $\mathbb{P}(\bigcap_{B \in \mathcal{B}} \overline B)\ge \prod_B(1-x_B)$.

    For a bad event $B$ of length $\ell$, set $x_B=2m^{-\ell}$. 
    Two bad events are dependent only if they use a common random variable $f_{ij}^\lambda(x)$. 
    A fixed bad event shares a random variable with only $O_{h,\Lambda,g}(m^{r-1})$ bad events of length $r$, and therefore $\sum_{B'\sim B}x_{B'}=\sum_{r=3}^{g}  O_{h,\Lambda,g}(m^{r-1}) \cdot 2 m^{-r}  =O_{h,\Lambda,g}(m^{-1})$. 
    Thus, for all sufficiently large $m$, $\prod_{B'\sim B}(1-x_{B'})\ge 1- \sum_{B'\sim B}x_{B'}\ge 1/2$. Hence $\mathbb{P}(B)\le m^{-\ell}\le x_B\prod_{B'\sim B}(1-x_{B'})$, so the conditions of the local lemma are satisfied. 
    Since $\sum_B x_B=O_{h,\Lambda,g}(1)$ and $\max_B x_B=o(1)$, the local lemma gives $\mathbb{P}(\bigcap_B \overline B)\ge\prod_B(1-x_B)\ge \prod_B e^{-2x_B}>0$, and hence $\mathbb{P}(\operatorname{girth}(G)>g)\ge c >0$, for all sufficiently large $m$. 
    Here, $c$ is some constant independent of $m$.

    It remains to show that $\mathbb{P}(\alpha(G)>m)\to 0$ as $m\to \infty$. 
    Let $A_i\subseteq V_i$, put $a_i=\abs{A_i}$, and set $A=A_1\cup\cdots\cup A_h$ and $a=a_1+\cdots+a_h$. 
    Suppose $a=m+1$.
    For fixed $A_1,\ldots,A_h$, independence of $A$ requires every directed edge from $A_i$ to $V_j$ to avoid $A_j$, for all $i\ne j$, and hence $\mathbb{P}(A\text{ is independent})=\prod_{i\ne j}(1-a_j/m)^{\Lambda a_i}  \le \prod_{i\neq j}\exp(-\frac{a_j}{m} \cdot \Lambda a_i)=\exp(-\frac{\Lambda}{m} \sum_{i\neq j} a_ia_j)$.

    First, assume $a_i\le m/2$ for every $i$. 
    We have $\sum_i a_i^2\le \frac{m}{2}\sum_i a_i=\frac{ma}{2}$, and hence $\sum_{i\ne j}a_ia_j=a^2-\sum_i a_i^2\ge a^2-\frac{ma}{2} >\frac{m^2}{2}$. 
    Therefore, $\mathbb{P}(A\text{ is independent})\le \exp(-\frac{\Lambda m}{2})$. 
    Since there are at most $2^{hm}$ choices for $A$, the probability of having an independent set $A$ with $\abs{A}=m+1$ and $\abs{A\cap V_i}\le m/2$ for all $i$ is at most $2^{hm}\exp(-\frac{\Lambda m}{2})=o(1)$, by our choice of $\Lambda$.

    Now, assume instead that $a_i>m/2$ for some $i$. 
    By symmetry it suffices to handle the case $a_1>m/2$.
    If $A$ is independent, every random edge directed from a selected vertex outside $V_1$ into $V_1$ must land in the complement $V_1\setminus A_1$, which has size $m-a_1$. 
    Hence, $\mathbb{P}(A\text{ is independent})\le (\frac{m-a_1}{m})^{\Lambda(a-a_1)}$. 
    The case $a_1=m$ has probability $0$, so we can ignore it from now on.
    Fix $\frac{m}{2}< a_1<m$. 
    There are $\binom m{a_1}$ choices for $A_1$. 
    Summing over all $\binom{(h-1)m}{a-a_1}$ possible choices of the remaining $a-a_1$ vertices outside $V_1$
    gives an upper bound
    \begin{equation}\label{eq:indepsetupperbound}
    \binom m{a_1}\binom{(h-1)m}{a-a_1}\left(\frac{m-a_1}{m}\right)^{\Lambda(a-a_1)}
    \end{equation}
    on the probability that an independent set~$A$ with $\abs{A\cap V_1}=a_1$ and $\abs{A}=m+1$ exists.
    Now, we deduce that 
    \begin{align*}
        \binom{m}{a_1}\binom{(h-1)m}{a-a_1} \left(\frac{m-a_1}{m}\right)^{\Lambda(a-a_1)}
        &\le \left(\frac{em}{m-a_1}\right)^{m-a_1}\left(\frac{e(h-1)m}{a-a_1}\right)^{a-a_1} \left(\frac{m-a_1}{m}\right)^{\Lambda(a-a_1)}\\
        &\leq \left(\frac{em}{m-a_1}\right)^{m-a_1}\left(\frac{ehm}{m-a_1}\right)^{a-a_1} \left(\frac{m-a_1}{m}\right)^{\Lambda(a-a_1)}\\ 
        &= \left(\frac{em}{m-a_1}\right)^{m-a_1}\left(\frac{ehm}{m-a_1}\right)^{m-a_1+1} \left(\frac{m-a_1}{m}\right)^{\Lambda(m-a_1+1)}\\ 
        &=  eh        (e^2h)^{m-a_1}  \left(\frac{m-a_1}{m}\right)^{(\Lambda-2)(m-a_1)} 
        \left(\frac{m-a_1}{m}\right)^{\Lambda-1}\\ 
        &\leq  eh        (e^2h)^{m-a_1}  \left(\frac12\right)^{(\Lambda-2)(m-a_1)} 
        \left(\frac{m-a_1}{m}\right)^{\Lambda-1}.\\ 
    \end{align*}
    Summing over all $m/2<a_1<m$ gives that the probability of having an independent set $A$ of size $m+1$  with $\abs{A\cap V_1}>m/2$ is at most
    \begin{equation}\label{eq:indepsetupperbound2}
        ehm^{1-\Lambda}\sum_{m/2<a_1<m}(m-a_1)^{\Lambda-1}(e^2h2^{2-\Lambda})^{m-a_1}.       
    \end{equation}
    Let $\theta = e^2 h2^{2-\Lambda}$.
    Then $\theta < 1$ by the choice of $\Lambda$, so the sum over $a_1$ in~\eqref{eq:indepsetupperbound2} is bounded from above by the convergent series $\sum_{k\ge1}k^{\Lambda-1}\theta^k$.
    Thus, the quantity in~\eqref{eq:indepsetupperbound2} is $O_{h,\Lambda}(m^{1-\Lambda})$, and hence tends to $0$ as $m \to \infty$.
    Multiplying by $h$ to account for the $h$ possible choices of an index $i$ for which $a_i>m/2$, it follows that the probability of having an independent set of size $m+1$ with $a_i>m/2$ for some $i$ is~$o(1)$. 
    Thus, with probability tending to $1$, there is no independent set of size greater than $m$. 
    Since each part $V_i$ is independent, we have $\alpha(G)\ge m$, and therefore $\mathbb{P}(\alpha(G)=m)\to 1$ as $m\to \infty$.

    Combining the estimates, for sufficiently large $m$, we have $\mathbb{P}(\operatorname{girth}(G)>g\text{ and }\alpha(G)=m)\ge c-o(1)>0$. 
    Thus, for every sufficiently large $m$, there is a desired $h$-partite graph $G$.
\end{proof}

\end{document}